\documentclass{article}

\usepackage{lmodern}			
\usepackage{amsmath}
\usepackage{amssymb}
\usepackage{amsthm}
\usepackage[latin1]{inputenc}
\usepackage[T1]{fontenc}      
\usepackage{geometry}  
\usepackage{array}   
\usepackage{ulem}    
\usepackage[english]{babel}		
\usepackage{hyperref}
\usepackage{color}

\usepackage{endnotes}
\let\footnote=\endnote
\makeatletter
\long\def\@endnotetext#1{%
\if@enotesopen \else \@openenotes \fi
\immediate\write\@enotes{\@doanenote{page \thepage}}
\begingroup
\def\next{#1}
\newlinechar='40
\immediate\write\@enotes{\meaning\next}%
\endgroup
\immediate\write\@enotes{\@endanenote}}
\makeatother

\newtheorem{thm}{Theorem}

\newtheorem{defn}[thm]{Definition}
\newtheorem{prop}[thm]{Proposition}
\newtheorem{cor}[thm]{Corollary}
\newtheorem{lem}[thm]{Lemma}
\newtheorem{rem}[thm]{Remark}
\newtheorem{exe}[thm]{Example}

\newtheorem*{thm*}{Theorem}
\newtheorem*{hyp*}{Hypothesis}
\newtheorem*{defn*}{Definition}
\newtheorem*{prop*}{Proposition}
\newtheorem*{cor*}{Corollary}
\newtheorem*{lem*}{Lemma}
\newtheorem*{rem*}{Remark}
\newtheorem*{exe*}{Example}
\newtheorem*{todo*}{Idea(s) to check}

\DeclareMathOperator*{\argmin}{arg\,min}
\DeclareMathOperator*{\argmax}{arg\,max}
\let\refOld\ref
\renewcommand{\ref}[1]{(\refOld{#1})}

\renewcommand{\phi}{\varphi}
\newcommand{\eps}{\varepsilon}

\newcommand{\supp}{\textrm{Supp }}

\usepackage{enumitem}

\newcommand{\newcomment}[1]{}
\newcommand{\verif}[1]{}
\newcommand{\Haus}{\textnormal{Haus}}

\renewcommand{\labelenumi}{\textnormal{(H\arabic{enumi})}}

\newcounter{compteur} 

\title{Mass localization}
\author{Thibaut Le Gouic}

\begin{document}


\maketitle

\begin{abstract}
For a given class $\mathcal{F}$ of closed sets of a measured metric space $(E,d,\mu)$, we want to find the smallest element $B$ of the class $\mathcal{F}$ such that $\mu(B)\geq 1-\alpha$, for a given $0<\alpha<1$. This set $B$ \textit{localizes the mass} of $\mu$. Replacing the measure $\mu$ by the empirical measure $\mu_n$ gives an empirical smallest set $B_n$. The article introduces a formal definition of small sets (and their size) and study the convergence of the sets $B_n$ to $B$ and of their size.
\end{abstract}
\tableofcontents

\section{Introduction}\label{trim:presentation}

The framework of our study is a measured metric space $(E,d,\mu)$. Mass localization intends to find in this setting a small Borel set $B$ such that $\mu(B)\geq 1-\alpha$ for some given $0<\alpha<1$.
The measure $\mu$ conditioned on $B$ is a new measure that we say to be $\alpha$-localized and denote $\mu_\alpha$.
This article provides a definition of a \textit{smallest} Borel set of probability $1- \alpha$ in order to obtain a localized version of the measure with the  \textit{smallest} support possible. 

This \textit{smallest} Borel set represents intuitively the "essential part" of the measure. However, it seems difficult to give an universal definition of "\textit{smallest}": although a ball centered on the origin as the smallest set with standard Gaussian measure on $\mathbb{R}^d$ seems a good choice, it is not obvious to define such set if the measure is not unimodal or if it is not symmetric or if it is not even defined on an Euclidian space.

Consistency is an important property we want for our notion. In statistics, the measure $\mu$ often unknown, is usually approximated by a sequence of probability measures $(\mu_n)_{n\geq 1}$. The \textit{smallest} closed set with $\mu_n$ -probability $1-\alpha$ should become closer to the \textit{smallest} one of $\mu$-probability $1-\alpha$ as $n$ grows.

Several methods have been studied in order to define such sets.

A first method is to choose a class $\mathcal{F}$ of subsets of $E$ partially ordered by their volume and to pick the smallest set (for this order) of this class with a $\mu$-probability greater than $1-\alpha$. This set corresponds to the level sets of a density function $f$ whenever $\mu$ is absolutely continuous with respect to the Lebesgue measure and the class $\mathcal{F}$ contains the level sets. An other way to define this set is to maximize
\begin{equation}\label{excess_mass}
\mu(B)-\beta\lambda (B),
\end{equation}
over $B \in \mathcal{F}$, where $\lambda$ is the Lebesgue measure and $\mu(\{f\geq \beta\})=1-\alpha$. 
This notion is known as \textit{excess mass}.
Denote by $B_\beta$ the maximizer of (\ref{excess_mass}) and by $B_\beta^n$ the maximizer of
\[
\mu_n(B)-\beta\lambda (B),
\]
for $\mu_n$ the empirical measure.
It is then of interest to determine if $B_\beta^n$ converges to $B_\beta$ and to exhibit a rate of convergence in this case.

The article \cite{Hartigan1987} considers the case of $\mathcal{F}$ being the set of all convex sets of $\mathbb{R}^2$ and proves that the Hausdorff distance $d_H(B_n,B)$ between $B_\beta^n$ and $B_\beta$ converges to $0$ and satisfies
\[
d_H(B_\beta^n,B_\beta)=O\left(\frac{\log n}{n}\right)^{2/7}.
\]
The article \cite{Nolan1991} considers sets $\mathcal{F}$ as  the set of all ellipsoids. Consistency of $B_\beta^n$ is proven, as well as the following limit theorem. Let $c_n$ and $c$ be the centers of the ellipsoids $B_\beta^n$ and $B_\beta$ respectively, and let $\sigma_n$ and $\sigma$ be the vector containing the entries of the matrix defining the ellipsoids $B_\beta^n$ and $B_\beta$ respectively, then, if the level sets of the measure $\mu$ are ellipsoids,
\[
n^{1/3}(c_n-c,\sigma_n-\sigma)
\]
is weakly converging to the maximum of a Gaussian process.
\cite{Polonik1997} studies a more general case, with a different notion of convergence, and showed the consistency  of $B_\beta^n$ for the pseudo-distance
\[
d_\mu(F,G)=\mu(F\triangle G),
\]
where $\triangle$ denotes the symmetric difference, whenever the class $\mathcal{F}$ is a Glivenko-Cantelli class. Under several hypotheses including that the level sets of the measure $\mu$ belongs to $\mathcal{F}$ and regularity conditions on $\mu$, the article obtains the following rate of convergence
\[
d_\mu(B_\beta,B_\beta^n) = O(n^{-\delta}),
\]
for a constant $\delta$ depending on the regularity of $\mu$.
This excess mass approach leads to rather precise results in many cases. However, it comes with few drawbacks, such as the condition that $\mathcal{F}$ must contains the level sets of the unknown measure $\mu$, which requires a certain knowledge on the $\mu$. Requirements on the regularity of $\mu$ can also be unsatisfactory for some applications. Also, this approach is restricted to the case of spaces with finite dimension (and often $\mathbb{R}^d$).


A second method comes from the notion of trimming on $\mathbb{R}$ extended to $\mathbb{R}^d$. On $\mathbb{R}$, the smallest set of $\mu$-probability $1-\alpha$ is defined as
\[
[F^{-1}(\alpha/2);F^{-1}(1-\alpha/2)],
\]
where $F$ is the cumulative distribution function of $\mu$. Replacing $F$ by the empirical cumulative distribution function $F_n$ defines the empirical smallest set. Extension to $\mathbb{R}^d$ can be done in the following way: $C_\alpha$ denotes the intersection of all the closed half spaces of $\mu$-probability greater than $1-\alpha$. $C_\alpha$ is then a non-empty convex set for $\alpha<1/2$, if the measure $\mu$ is regular enough. \cite{Nolan92} deals with the rate of convergence of $C_n$, defined similarly with the empirical measure $\mu_n$ and shows its consistency. In order to quantify the rate of convergence of $C_n$ to $C_\alpha$, the article introduce the following random functions
\[
r_n(u)=\inf \{r\geq 0;ru \notin C_n\},
\]
and
\[
r_\alpha(u)=\inf \{r\geq 0;ru \notin C_\alpha\},
\]
and establishes the weak convergence to a Gaussian process defined on the unit sphere $\mathbb{S}^{d-1}$ of the process
\[
\sqrt{n}(r_n-r_\alpha),
\]
under regularity conditions on the density function of $\mu$.

The article \cite{kmeans} presents another method, called $\alpha$-trimmed $k$-means, which introduces very few arbitrary parameters. This method chooses the support of the $\alpha$-localized measure $\nu$ as the one minimizing the distortion to its best $k$-quantifier. Formally, for a given function $\Phi$, and a given integer $k$, the method consists in choosing
\[
B_\alpha \in \argmin \left\{\inf_{\{m_1,...,m_k\}\subset \mathbb{R}^d}\int_B \Phi \left(\inf_{1\leq i \leq k} \|X - m_i\| \right) d\mu; \mu(B)\geq 1-\alpha\right\}.
\]
After proving the existence of such minimizer, the article \cite{kmeans} shows the consistency of $B_\alpha$: if $(\mu_n)_{n\geq 1}$ weakly converges to an absolutely continuous measure $\mu$ then, for any choice of 
\[
B_\alpha^n \in \argmin \left\{\inf_{\{m_1,...,m_k\}\subset \mathbb{R}^d}\int_B \Phi \left(\inf_{1\leq i \leq k} \|X - m_i\| \right) d\mu_n; \mu_n(B)\geq 1-\alpha\right\},
\]
the sequence $(B_\alpha^n)_{n\geq 1}$ converges to $B_\alpha$ (when unique) for the Hausdorff metric. Theses results hold on $\mathbb{R}^d$.

The main goal of our article is to provide a new definition that is intuitive and avoid usual hypotheses, that remains consistent.

\subsection{Definitions}

We define a notion of \textit{smallest} closed set and introduce some properties that will help to understand its meaning.
The framework of the definition aims to be fairly general.
$(E,d)$ is a Polish space (metric, separable and complete space) and $\mu$ is a Borel measure on $(E,d)$.
A \textit{smallest} set will be defined as the minimizer of a function $\tau$ defined on a class $\mathcal{F}$ of closed subsets of $E$.

\subsubsection{Stable set}

In order to ensure the existence of the smallest set in a class $\mathcal{F}$ of sets, the class needs to be stable in some way. The following definition of such stability will be an assumption made on the class. Let first set the following notation.

For a given set $B$ and $\eps>0$, the set $B^\eps$ is the $\eps$\textbf{-neighborhood} of $B$:
\[
B^\eps:=\{x\in E; \exists y \in B, d(x,y)< \eps\}.
\]

\begin{defn}[Stable set]\label{def:ensemblestable}
Let $(B_n)_{n \geq 1}$ be a sequence of closed sets, denote $\lim_n B_n$ the set
\[
\lim_n B_n := \bigcap_{\eps > 0} \bigcup_{k\geq 1} \bigcap_{n \geq k} B_n^\eps.
\]
Let $\mathcal{F}$ be a class of closed sets of $E$. $\mathcal{F}$ is \textbf{stable} if $E\in \mathcal{F}$ and
\[
(B_n)_{n \geq 1} \subset \mathcal{F} \implies \exists (n_k)_{k\geq 1}, n_k \rightarrow \infty, \lim_k B_{n_k} \in \mathcal{F}.\\
\]
\end{defn}

This notion of stability is close to the completeness under Hausdorff convergence. Indeed, it is strictly equivalent if the metric space $(E,d)$ is compact, as it will be discussed in the next remarks. 

As we defined $\mathcal{F}$ as a subset of the closed sets of $(E,d)$, we first check that our notion of stability makes sense for a class of closed sets.

\begin{rem}\label{rem:lim_ferme}
Given a sequence of sets $(B_n)_{n\geq 1}$, $\lim_n B_n$ is always closed.
Indeed, denote $B(x,\eps/2)$ the ball centered in $x$ of radius $\eps/2$,
\begin{align}
x \notin \lim_n B_n &\Leftrightarrow \exists \eps>0, \forall k\geq 1, \exists n\geq k, x \notin B_n^\eps\\
&\Rightarrow \exists \eps>0, \forall k\geq 1, \exists n\geq k, B(x,\eps/2) \cap B_n^{\eps/2} = \emptyset\\
& \Rightarrow \exists \eps>0, B(x,\eps/2) \cap \lim_nB_n = \emptyset.
\end{align}
In other words, $(\lim_nB_n )^c$ is open, and $\lim_nB_n$ is thus closed.
\end{rem}

The following
remark aims to clarify stability.

\begin{rem}\label{rem:generalizedH}
When $(B_n)_{n\geq 1}$ is converging to $B_\infty$ for the Hausdorff metric, then
\[
\lim_n B_n = B_\infty.
\]
Indeed, denote $\eps_k$ the smallest $\eps>0$ such that $B_n \subset B^\eps_\infty$ and $B_\infty \subset B^\eps_n$ for all $n \geq k$, then,
\[
B_\infty = \bigcap_{\eps > 0} \bigcup_{k\geq 1} \bigcap_{n \geq k} B_\infty \subset \bigcap_{\eps > 0} \bigcup_{k\geq 1} \bigcap_{n \geq k} B_n^{\eps_k} = \lim_n B_n \subset \bigcap_{\eps > 0} \bigcup_{k\geq 1} \bigcap_{n \geq k} B_\infty^{\eps + \eps_k} = B_\infty.
\]

In a more general setting, given a sequence of closed balls $(K_k)_{k\geq 1}$ such that $\cup_{k\geq 1}K_k=E$, and given a sequence $(B_n)_{n\geq 1}$, if there exists $B_\infty$ such that for any $k\geq 1$, the sequence $(B_n\cap K_k)_{n \geq 1}$ converges in Hausdorff metric to $B_\infty\cap K_k$, then
\[
\lim_n B_n = B_\infty.
\]
\end{rem}

\begin{rem}\label{rem:totalborne}
In a metric space $(E,d)$ such that every bounded closed set is compact (this is the case for instance, of locally compact length spaces), it is easier to understand the meaning of stability of a class.
For any sequence $(B_n)_{n \geq 1}$ of closed sets of $E$, and any sequence $(K_k)_{k\geq 1}$ of increasing closed balls such that $\cup_{k\geq 1}K_k=E$, there exist a set $B_\infty$ and a subsequence (relabeled $(B_n)_{n\geq 1}$) such that $B_n\cap K_k$ converges in Hausdorff metric to $B_\infty$ and 
\[
\lim_n B_n = B_\infty.
\]

In this case, a stable class in the sense of definition \ref{def:ensemblestable} is just a compact class for the Hausdorff convergence on large balls.
Indeed, in such spaces $E$, there exists an increasing sequence of compacts $(K_k)_{k \geq 1}$ such that $\cup_{k\geq 1} K_k = E$, take for instance a sequence of balls centered on the same point, with an increasing radius. 
The Hausdorff convergence on large balls is then equivalent to the Hausdorff convergence of $(B_n\cap K_k)_{n \geq 1}$ for any $k \in \mathbb{N}$.
Since the closed sets in $K_k$ forms a compact class for the Hausdorff convergence, there exists a subsequence of $(B_n\cap K_k)_{n \geq 1}$ converging to some $B_\infty^k$.
Using diagonal argument, we may extract a subsequence of the original sequence $(B_n)_{n \geq 1}$ such that for any $k \in \mathbb{N}$, $(B_n\cap K_k)_{n \geq 1}$ converges in Hausdorff metric to $B_\infty^k$.
It is easily checked that $B_\infty := \cup_k B_\infty^k$ is a limit of a subsequence of $(B_n)_{n \geq 1}$, in the sense of definition \ref{def:ensemblestable}.
\end{rem}

Let us introduce some examples of stable sets.
\begin{exe}
Take $\mathcal{F}$ as the set of all closed sets. Stability is then obvious since the limit considered in the definition of a stable set is always closed as shown in the remark \ref{rem:lim_ferme}.
\end{exe}
\begin{exe}
The set of all balls is generally not stable, but is does not take much to make is stable. The set of all closed balls and half spaces in $\mathbb{R}^d$ is a stable class.
This assertion can be proved using parametrization of the center of the balls in spherical coordinates and using compactness of spheres.
\end{exe}
\begin{exe}
Other shapes of sets of $\mathbb{R}^d$ make stable classes. Ellipsoids, rectangles, or convex bodies with bounded diameter (by some fixed $R<\infty$) all form stable classes.
And it is possible to get rid of the bounded diameter by adding some sets to the class.
\end{exe}
\begin{exe}\label{ex:eps-stable}
If $\mathcal{F}$ is a stable class of convex sets of a metric space $(E,d)$ such that closed balls are compacts, then
\[
\mathcal{F}_\eps := \{ \cup_{F \in \mathcal{G}} F ; \mathcal{G} \subset \mathcal{F}, \forall F,G \in \mathcal{G} \inf_{x \in F, y \in G} d(x,y) \geq \eps \},
\]
is also a stable class (see lemma \ref{lem:stable}).
\end{exe}

\subsubsection{Size function}

As we aim to define a smallest set of the sets of $\mathcal{F}$, we need to define a notion of size. 
This is done using a function $\tau$, meant to measure the \textit{size} of a set.
In order to localize the mass, we will thus minimize the size of a set, among all sets given a probability measure.

In order to express our assumptions on $\tau$, we first define the \textit{Hausdorff contrast}.

\begin{defn}[Hausdorff contrast]\label{def:contrasteHaus}\sloppy
Let $A$ and $B$ be two closed subset of a Polish space $(E,d)$. The \textbf{Hausdorff contrast} between $A$ and $B$ is defined by
\[
\Haus(A|B) := \inf\{\eps>0|A\subset B^\eps\}.
\]
\end{defn}

We can then remark that the Hausdorff metric $d_H(A,B)$ between two closed sets $A$ and $B$ is then
\[
d_H(A,B)=\Haus(A|B)\vee\Haus(B|A).
\]

We now define formally a size function.
\begin{defn}[Size function]
Let $(E,d)$ be a metric space. A function $\tau:\mathcal{F}\rightarrow \mathbb{R}^+$ is called a \textbf{size function} if it
satisfies the three following conditions:
\begin{enumerate}
\sloppy
\item $\tau$ is increasing, i.e. $A\subset B \implies \tau(A) \leq \tau(B)$, \label{increasing}
\item for any decreasing sequence $(A_n)_{n \geq 1} \subset \mathcal{F}$ such that $\tau(A_1)<\infty$ and $\Haus(A_n |\cap_k\nolinebreak A_k ) \rightarrow \nolinebreak 0$, the following holds $\tau(A_n) \rightarrow \tau(\cap_n A_n)$,\label{monotony}
\item for any sequence $(A_n)_{n\geq 1} \subset \mathcal{F}$, $\tau(\lim_n A_n) \leq \liminf_n \tau(A_n)$. \label{lsc}
\setcounter{compteur}{\value{enumi}} 
\end{enumerate}
\end{defn}

Hypothesis (H\ref{monotony}) on the size function requires some Hausdorff contrast.
This particular choice make the hypothesis weaker and allow the hypothesis to hold for size function that give finite size to non compact sets.
The consequences of these hypothesis will be more detailed in the sequel of the paper.

\subsection{Overview of the main result}

Our main result states that under the condition (H\ref{increasing}), (H\ref{monotony}) and (H\ref{lsc}), for the empirical measure $\mu_n$, and a stable class $\mathcal{F}$,
\[
\tau^\alpha \leq \liminf_n \tau_n^\alpha \leq \limsup_n \tau_n^\alpha \leq \lim_\eps \tau^{\alpha-\eps},
\]
where $\tau_n^\alpha=\min\{\tau(B);B\in \mathcal{F},\mu_n(B)\geq 1-\alpha\}$.

It implies the convergence of $\tau_n^\alpha$ when $. \mapsto \tau^.$ is continuous.

The result actually holds for a wider class of sequence of measures $(\mu_n)_{n\geq 1}$. 

Moreover, simple conditions on the sequence imply the convergence of the minimizers of the $\tau_n^\alpha$ for different metrics (depending on the conditions assumed). This is discussed in the next sections.

\section{First properties}

\subsection{Existence}

Let us recall the setting. $(E,d)$ is a Polish space and $\mu$ is a Borel probability measure on $(E,d)$. Given a size function $\tau$, a stable class $\mathcal{F}$ of closed sets of $E$, and a level $\alpha$, we define the support $B^\alpha$ of the $\alpha$-localized measure $\mu^\alpha$ of $\mu$ by - when possible:
\[
B^\alpha \in \argmin \{\tau(A); A \in \mathcal{F}, \mu(A)\geq 1-\alpha\},
\]
and set
\[
\mu^\alpha=\mu(.|B^\alpha).
\]
Our first concern is whether $B^\alpha$ exists. It is the matter of the next result. 

\begin{thm}[Existence of a  minimum]\label{existence}
Let $(E,d)$ be a Polish space, $\mathcal{F}$ a stable class and $\mu$ a probability measure on $(E,\mathcal{B}(E))$. Set $0<\alpha<1$. Suppose (H\ref{lsc}). Then, there exists $B\in \mathcal{F}$ such that
\[
B \in \argmin \left\{\tau(A);A \in \mathcal{F}(E), \mu(A)\geq 1-\alpha \right\}.
\]
\end{thm}

\begin{rem}
Hypothesis (H\ref{lsc}) can not just be omitted. Indeed, if $\tau(B)$ is defined as the Lebesgue measure of the closure of $B$ on $\mathbb{R}^d$, take
\[
\mu= \alpha \gamma_d + (1-\alpha) q,
\]
where $q$ is a probability measure supported on $\mathbb{Q}^d$ and $\gamma_d$ is the standard Gaussian measure on $\mathbb{R}^d$, then, the sequence $(B_n)_{n \geq 1}$ defined by
\[
B_n:=\{x_k\}_{1 \leq k \leq n} \cup B(0,r_n),
\]
with $\{x_n\}_{n \geq 1}=\mathbb{Q}^d$ and $r_n \rightarrow 0$ so that $\mu(B_n)=1-\alpha$, is a minimizing sequence. And $\tau(B_n)=\tau\left(B(0,r_n)\right)$ so that $\tau^\alpha = 0$ but $\tau\left(\mathbb{Q}^d\right)=+\infty$.
\end{rem}

The minimizer is not necessarily unique. This seems natural with the following example. Take $\mu$ as the uniform law on the unit square and an isometric $\tau$. Then any translation small enough of the minimizer will have the same size and the same measure, and will thus be another minimizer.
Another result (corollary \ref{corHausdorffconsistency}) will comfort us proving that minimizers form a compact set for Hausdorff metric.

The stability condition on $\mathcal{F}$ is needed for existence of the minimum. However, it can be lightly weakened.

\begin{rem}[On stability of $\mathcal{F}$]\label{rem:faible-stable}
Since the minimal size $\min\{\tau(A);A\in \mathcal{F},\mu_n(A)\geq 1-\alpha\}$ is bounded if $\tau^\alpha=\min\{\tau(A);A\in \mathcal{F},\mu(A)\geq 1-\alpha\}<\infty$, then we may suppose instead of stability of $\mathcal{F}$ that all the classes
\[
\mathcal{F}^M:=\mathcal{F}\cap \{A; \tau(A)\leq M\}
\]
for $M<\infty$ are stable. It is a weaker notion since $\tau(\lim B_n) \leq \liminf\tau(B_n)$ for any sequence $(B_n)_{n \geq 1}$ in $\mathcal{F}$, under (H\ref{lsc}).
\end{rem}

\subsection{\texorpdfstring{Regularity of $\tau$}{Regularity of tau}}
Denote
\[
\tau^\alpha=\inf \{\tau(A); A \in \mathcal{F}, \mu(A)\geq 1-\alpha, B\}.
\]
It seems natural to expect $\alpha \mapsto \tau^\alpha$ to be continuous when $\mu$ is regular enough. It also seems natural, for instance, to have $B^\alpha$ growing continuously when $\alpha$ decreases to zero, for a unimodal measure $\mu$. 
This is the concern of this paragraph, the first one establishing the right continuity.

\begin{prop}[Right continuity]\label{rightcont}
Let $(E,d)$ be a Polish space, $\mu$ a probability measure on $(E,\mathcal{B}(E))$ and $\mathcal{F}$ a stable class. Let $0<\alpha<1$. Then, under (H\ref{lsc}), $\alpha \mapsto \tau^\alpha$ is right continuous.
\end{prop}

The continuity will require some more hypotheses as shows the following example of discontinuity. 
Take $\mu = (\delta_x + \delta_y)/2$ and $\alpha=1/2$, and it is not difficult to find some $\tau$ that is not continuous on $\alpha$.

Thus, it is clear that continuity property of this function needs regularity on the measure we want to localize, with respect to the class $\mathcal{F}$. 
This is why we introduce the notion of $\mathcal{F}$-regularity.

\begin{defn}[$\mathcal{F}$-regularity]
A probability measure $\mu$ is said to be \textbf{$\mathcal{F}$-regular} if for all $B \in \mathcal{F}$, any $\delta>0$ and any $C \in \mathcal{F}$ such that $B \subset C$ and $\mu(B) < \mu(B^\delta \cap C)$, there exists $A \in \mathcal{F}$ such that
\begin{align*}
A \subset B^\delta \cap C, \\
\mu(B) < \mu(A).
\end{align*}
\end{defn}

The only purpose of this notion is the continuity of the application $\alpha \mapsto \tau^\alpha$. It is restrictive on $\mu$ only when $\mathcal{F}$ is not rich enough.
Taking $\mathcal{F}$ as the class of all closed sets of $E$ make any probability measure $\mathcal{F}$-regular. Indeed, since $\mu(B^\delta\cap C) = \lim_n \mu(B^{\delta-1/n}\cap C)$, there exists $n \geq 1$ such that $\mu(B) < \mu(B^{\delta - 1/n} \cap C)$ and then we can choose $A:=\overline{B^{\delta}} \cap C$.
On the other hand, if $\mathcal{F}$ is not rich enough so that $\tau(\mathcal{F})$ is not even connected, it is easy to build a measure $\mu$ that is not $\mathcal{F}$-regular.

\begin{prop}[Continuity]\label{continuity}
Let $\mu$ be a probility measure on a Polish space $(E,d)$. Suppose (H\ref{increasing}), (H\ref{monotony}) and (H\ref{lsc}), and that $\mu$ is $\mathcal{F}$-regular, has a connected support and that $\tau^\alpha$ is finite for any $\alpha>0$ then, the mapping $\alpha \mapsto \tau^\alpha$ is continuous.
\end{prop}

\begin{rem}
The condition $\tau^\alpha<\infty$ just avoids a degenerated case.
\end{rem}

This continuity condition is a first step toward the main matters of our article, the consistency.

\section{Consistency}

\subsection{\texorpdfstring{$\tau$-tightness}{tau-tightness}}

In order to show the consistency of the mass localization when a sequence of measures $(\mu_n)_{n\geq 1}$ converges to a measure $\mu$, we must make some assumptions on the sequence of measures. The first and most important hypothesis for consistency is the $\tau$-tightness.

\begin{defn}[$\tau$-tightness]
A sequence of random probability measures $(\mu_n)_{n \geq 1}$ almost surely weakly converging to a measure $\mu$ is \textbf{$\tau$-tight} if for any $\delta >0$ and any $B \in \mathcal{F}$ such that $\tau(B)<\infty$, almost surely, for any $C\in \mathcal{F}$ such that $B\subset C$ and $ \mu(B)\leq\liminf_n \mu_n(C)$, there exists $A \in \mathcal{F}$ such that
\begin{align*}
\mu(B) \leq \liminf_n \mu_n(A), \\
B \subset A \subset B^\delta \cap C,\\
\tau(A)<\infty.
\end{align*}
\end{defn}

An important remark on this definition is that a $\tau$-tight sequence of random measures does not have necessarily almost surely $\tau$-tight realizations.
This can happen to empirical measures for instance.
This subtlety lies in the position of "almost surely" in the definition, that is, after the choice of $B$ and $\delta$ made.

We can also remark the following.
Inequality $\mu(B)\leq\liminf_n \mu_n(C)$ is not a consequence of $B\subset C$.
Indeed, the portmanteau theorem states $\limsup_n \mu_n(C) \leq \mu(C)$ and $\limsup_n \mu_n(B) \leq \mu(B)$ since $B$ and $C$ are closed.
The conditions for $\tau$-tightness on $B \in \mathcal{F}$ such that $\mu(B)=\lim_n \mu_n(B)$ is clearly verified for $A:=B$.
The definition of $\tau$-tightness can be understood as follows.
Whenever $(\mu_n)_{n\geq 1}$ does not catch all the $\mu$-mass of $B$ (i.e. $\liminf_n \mu_n(B) < \mu(B)$) but some set $C$ that contains $B$ has its $\mu$-mass well caught (i.e. $\mu(B)\leq\liminf_n \mu_n(C)$), then $\mathcal{F}$ must have an element $A$ that also have its $\mu$-mass well caught (i.e. $\mu(B) \leq \liminf_n \mu_n(A)$), of finite size (i.e. $\tau(A) < \infty$) and that is stuck between $B$ and a $\delta$-neighborhood of $B$ intersected with $C$, for small $\delta$.

The following proposition states that this notion is not empty, and includes the empirical measures.

\begin{prop}[$\tau$-tightness of the empirical measure]\label{tauemp}
Let $\mu$ be a probability measure on $E$ such that $\tau^\alpha<\infty$ for $0 < \alpha < 1$. 
Let $(X_i)_{i\geq 1}$ be a sequence of i.i.d. random variables with common law $\mu$. Set $\mu_n=\frac{1}{n}\sum_{1 \leq i \leq n} \delta_{X_i}$. Then, $(\mu_n)_{n\geq 1}$ is $\tau$-tight.
\end{prop}

The empirical measure is actually not the only simple example of $\tau$-tight sequence. The following corollary gives a simple condition for a sequence of random probability measures to be $\tau$-tight.

\begin{cor}\label{cor:tau_tension}
Let $(\mu_n)_{n\geq 1}$ be a sequence of random probability measure on $E$ almost surely weakly converging to some measure $\mu$, such that $\tau^\alpha<\infty$, for any $0 < \alpha < 1$. 
If for all $B\in \mathcal{F}$, almost surely, $\mu(B) \leq \lim_n \mu_n(B)$, then $(\mu_n)_{n \geq 1}$ is $\tau$-tight.
\end{cor}

This corollary says that $\tau$-tightness is implied by almost sure convergence of $\mu_n(B)$ for each $B\in \mathcal{F}$ and thus, dropping the "almost sure" makes the $\tau$-tightness much more restrictive.

We can now state our first result on consistency.

\subsection{\texorpdfstring{$\tau$-consistency}{tau consistency}}

Our goal is to show that when $\mu_n$ converges to $\mu$, the size $\tau_n^\alpha$ of the smallest element of a given class $\mathcal{F}$ with $\mu_n$-mass at least $1-\alpha$ converges to the size $\tau^\alpha$ of the smallest of $\mu$-mass at least $1-\alpha$.
In other words, we want to prove consistency of the smallest size $\tau^\alpha$.
The following theorem states conditions for this consistency to hold.

\begin{thm}[Consistency]\label{consistency}
Let $(E,d)$ be a Polish space, $\mathcal{F}$ a stable class and $(\mu_n)_{n \geq 1}$ a $\tau$-tight sequence of random probability measures on $(E,\mathcal{B}(E))$ almost surely weakly converging to some measure $\mu$. Set $0<\alpha<1$. Choose any $B_n^\alpha \in \argmin \{\tau(A); A \in \mathcal{F}, \mu_n(A)\geq 1-\alpha\}$  and $\mu_n^\alpha=\mu_n(.|B_n^\alpha)$, for all $n\geq 1$.
Then, under hypotheses (H\ref{increasing}), (H\ref{monotony}) and (H\ref{lsc}), the sequence $(\mu_n^\alpha)_{n\in \mathbb{N}}$ is almost surely totally bounded for the weak convergence topology and $B_\infty^\alpha:=\lim_k B_{n_k}$ along any converging subsequence $(\mu_{n_k})_{k \geq 1}$ of $(\mu_n)_{n \geq 1}$ satisfies $\mu(B_\infty^\alpha) \geq 1- \alpha$ and almost surely
\[
\tau^\alpha \leq \tau(B_\infty^\alpha) \leq \liminf_{n \rightarrow \infty} \tau(B_n^\alpha) \leq \limsup_{n \rightarrow \infty} \tau(B_n^\alpha) \leq \lim_{\eps \rightarrow 0^+} \tau^{\alpha-\eps}.
\]

Moreover, if $\mu$ is $\mathcal{F}$-regular and its support is connected, the five terms above are equal.
\end{thm}

Note that the $\tau$-tightness condition is required only for the last inequality.

It is rather clear the if $\alpha \mapsto \tau^\alpha$ is not continuous for the measure $\mu$, we can hardly expect consistency of the smallest size $\tau^\alpha$. This first step of consistency brings us to consider consistency of the smallest set of the class itself.

\subsection{Minimizer consistency}

The smallest set in $\mathcal{F}$ with $\mu$-mass greater than $1-\alpha$ is not always unique, and therefore consistency does not just mean that minimizer for $\mu_n$ converges to the minimizer for $\mu$.
In order to give a sense to consistency, we will consider the set of all minimizers and the Hausdorff contrast between sets of elements of $\mathcal{F}$ (for some underlying metric $d_\mathcal{F}$ on $\mathcal{F}$).
We thus first recall the definition of Hausdorff contrast.
Let $A$ and $B$ be two sets. The Hausdorff contrast between $A$ and $B$ is defined by
\[
\Haus(A|B) := \inf\{\eps>0|A\subset B^\eps\}.
\]

Let us denote, for $0 < \alpha < 1$, a sequence of measures $(\mu_n)_{n\geq 1}$ and a measure $\mu$;
\[
\mathcal{S}^\alpha_n = \argmin\{\tau(A); A \in \mathcal{F}, \mu_n(A) \geq 1-\alpha\},
\]
and
\[
\mathcal{S}^\alpha = \argmin\{\tau(A); A \in \mathcal{F}, \mu(A) \geq 1-\alpha\}.
\]
The sets $\mathcal{S}^\alpha$ and $ \mathcal{S}^\alpha_n$ are thus two subsets of $\mathcal{F}$.
What we want is to find conditions under which
\[
\Haus(\mathcal{S}^\alpha_n|\mathcal{S}^\alpha) \rightarrow 0,
\]
when $n$ tends to infinity.

We now state and comment briefly the two hypotheses that will be made for our main result.

\begin{enumerate}
\setcounter{enumi}{\value{compteur}} 
\item $\cdotp \mapsto \tau^\cdotp$ is continuous at $\alpha$, \label{taucont}
\item $\forall (A_n)_{n\geq 1} \subset \mathcal{F}$ such that $\tau(\lim_n A_n)<\infty$, \\
\[
\lim\tau(A_n) = \tau(\lim A_n) \implies d_\mathcal{F}(A_n,\lim_k A_k) \rightarrow 0
\] \label{d-cont}
\setcounter{compteur}{\value{enumi}}
\end{enumerate}

where $d_\mathcal{F}$ denotes a metric on $\mathcal{F}$. A typical example of such metric is the Hausdorff metric or the measure of symmetric difference.
Section \ref{sec:sizeEx} is devoted to these examples and conditions that imply (H\ref{d-cont}).

The continuity condition (H\ref{taucont}) is a consequence of the proposition \ref{continuity}: a connected support for an $\mathcal{F}$-regular measure suffices.


We can now state a direct consequence of theorem \ref{consistency} and hypotheses (H\ref{taucont}) and (H\ref{d-cont}).

\begin{thm}[Consistency of the minimizers]\label{Hausdorffconsistency}
Let $(E,d)$ be a Polish space and $(\mu_n)_{n \geq 1}$ a sequence of random $\tau$-tight probability measures on $(E,\mathcal{B}(E))$ almost surely weakly converging to some $\mathcal{F}$-regular measure $\mu$. Set $0<\alpha<1$. Equip $\mathcal{F}$ with a metric $d_\mathcal{F}$

Suppose (H\ref{monotony}), (H\ref{lsc}), (H\ref{taucont}) (for the measure $\mu$), and (H\ref{d-cont}).

Then, almost surely,
\[
\Haus(\mathcal{S}^\alpha_n|\mathcal{S}^\alpha)\rightarrow 0.
\]
\end{thm}

We now make some remarks on the necessity of the hypotheses.

\begin{rem}[On (H\ref{taucont})]
Continuity condition on $\cdotp \mapsto \tau^\cdotp$ at $\alpha$ can not be dropped. Indeed, if we choose $d\mu_n= (1_{[0;1-\alpha-1/n]} + 1_{[2;2+\alpha+1/n]}) d\lambda$, then, taking for instance $\tau(B)=\int_0^1 \mathcal{M}(B,t)tdt$, yields to the minimizers $\supp \mu_n^\alpha= [0;1-\alpha-1/n] \cup [2;2+1/n]$ and $\supp \mu_\infty^\alpha = [0;1-\alpha] $.
\end{rem}

The result is stated with Hausdorff contrast and not Hausdorff metric. One can wonder what happens with the Hausdorff metric.

\begin{rem}[Convergence with Hausdorff metric]\label{converse}
The convergence of $\Haus(\mathcal{S}^\alpha|\mathcal{S}^\alpha_n)$ (i.e. the other contrast), is not always true.
Take for instance $\tau(B)=\int_0^1 \mathcal{M}(B,t)t dt$ (see definition \ref{def:packing}) and $d\mu_n=f_n d\lambda$ on $\mathbb{R}$, with $f_n(x)=1 + x/n$ for $-1/2<x<1/2$ and $f(x)=0$ otherwise. Then, the minimizer $[-1/2;1/2-\alpha]$ for $\mu=\mathbf{1}_{[0;1]}$ is not included in any $\eps$-neighborhood of $[-1/2+\alpha;1/2]$.
\end{rem}

When $\mathcal{S}^\alpha$ is a singleton, theorem \ref{Hausdorffconsistency} then states that any $B_n^\alpha$ converges to the minimizer for the limit $\mu$.
When $\mathcal{S}^\alpha$ is not a singleton, it states that all $B_n^\alpha$ in $\mathcal{S}^\alpha_n$ gets close to an element of $\mathcal{S}^\alpha$, uniformly. 
However, the remark \ref{converse} precises that there could be some elements of $\mathcal{S}^\alpha$ that will not be approximated.

We can derive from the proof of this result the following corollary.

\begin{cor}\label{corHausdorffconsistency}
Let $(E,d)$ be a Polish space and $\mu$ a probability measure on $(E,\mathcal{B}(E))$ such that $\tau^\alpha<\infty$. Set $0<\alpha<1$. Suppose (H\ref{monotony}), (H\ref{lsc}) and (H\ref{d-cont}). 
Then $\argmin \{\tau(A); A \in \mathcal{F}, \mu(A)\geq 1-\alpha\}$ is compact for the Hausdorff metric topology.
\end{cor}

The tools developed in the proofs of these results were also effective to prove some continuity on the minimizer on the level $\alpha$.

\subsection{Minimizer continuity}

The proof of the theorem of the minimizer consistency is based on the two lemmas \ref{lemHausdorff} and \ref{contrasteHaus}.
The same technique of proof leads to the following result.

\begin{prop}\label{contHaus}
Let $(E,d)$ be a Polish space and a $\mathcal{F}$-regular probability measure $\mu$ on $(E,\mathbb{B}(E))$. Set $0 < \alpha <1$ and $(\alpha_n)_{n \geq 1}$ converging to $\alpha$. Suppose (H\ref{monotony}), (H\ref{lsc}), (H\ref{taucont}), and (H\ref{d-cont}). Then,
\[
\Haus(\mathcal{S}^{\alpha_n}|\mathcal{S}^\alpha)\rightarrow 0.
\]
\end{prop}

In the case of a unique minimizer for all $\alpha$, this result states the continuity of the function that associate $\alpha$ to the minimizer for $\alpha$:
\[
\alpha \mapsto B^\alpha.
\]

\section{Examples}

In this section, we introduce examples of stable classes and size functions, in order to show the scope of our results. 

\subsection{Examples of stable classes}

\subsubsection{Closed sets}
The simplest example of stable class is the set of all closed sets. Indeed, given a sequence of closed sets $(A_n)_{n\geq 1}$, the limit $\lim A_n$ is also closed (see remark \ref{rem:lim_ferme}).

\subsubsection{Parametrized classes}

On $\mathbb{R}^d$, the set of closed balls and half spaces is also a stable class.
Indeed, each ball can be parametrized by $(x,r)$, the center and radius of the ball, and half spaces are limit (in the limit of sets we defined) of balls. 
From this parametrization, one can show that there exists a converging subsequence to any sequence of balls and half spaces in the sense we defined for the limit of sets.

More generally, one can use the remark \ref{rem:faible-stable} which states that for a given size function $\tau$, the condition on the class $\mathcal{F}$ to be stable can be weakened to the condition that $\mathcal{F}^M=\mathcal{F}\cap\{B;\tau(B)\leq M\}$ is stable for each $M\in \mathbb{R}^+$.
Then, taking any parametrized class $\mathcal{F}$ such that convergence of the parameters implies convergence of the sets (in the sense defined for stable sets) and such that $\mathcal{F}^M$ is compact for the Hausdorff metric gives a stable set, since Hausdorff convergence implies convergence in the sense we defined for stable sets.

\subsubsection{$\eps$-separated unions}

Another example of stable set is the one of $\eps$-separated union of elements of a stable class of convex sets (see lemma \ref{lem:stable}).
For $\mathcal{F}$ a stable class of convex sets of a metric space $(E,d)$ such that bounded sets are compact, the following set is stable
\[
\mathcal{F}_\eps := \{ \cup_{F \in \mathcal{G}} F ; \mathcal{G} \subset \mathcal{F}, \forall F,G \in \mathcal{G} \inf_{x \in F, y \in G} d(x,y) \geq \eps \}.
\]

This sets give an application to classification.

\subsection{Examples of size function}\label{sec:sizeEx}

\subsubsection{Packing}

Our first example of size function is functions depending on the packing of sets. 
Let us recall the notion of packing and the more common one of covering (see \cite{kolmogorov1961}).

\begin{defn}[t-covering]\label{def:recouvrement}
A set $\{B_i\}_{i\in I}$ of subsets of $E$ is a \textbf{$t$-covering} of a set $B$ if the diameter of any $B_i$ does not exceed $2t$ and
\[
B \subset \bigcup_{i \in I} B_i.
\]
The cardinal of the smallest $t$-covering of $B$ is then called \textbf{covering number} of $B$ for $t$ and is denoted $\mathcal{N}(B,t)$.
\end{defn}

The logarithm of $\mathcal{N}(B,t)$ is sometimes called the metric entropy of $B$, or its Kolmogorov entropy.

\begin{defn}[t-separated]\label{def:packing}
A set $B$ is \textbf{$t$-separated} if the distance of every two distinct points of $B$ is strictly greater then $t$.
The cardinal of the greatest $t$-separated subset of $B$ is called \textbf{packing number} of $B$ for $t$ and is denoted $\mathcal{M}(B,t)$.
\end{defn}

The logarithm of $\mathcal{M}(B,t)$ is sometimes called capacity of $B$.

These two notions are intuitively linked and carry information on the size of the set. 
The following proposition due to \cite{kolmogorov1961} compares $\mathcal{N}$ and $\mathcal{M}$.
\begin{prop}\label{compar}
For any set $B$ of a metric space $(E,d)$,
\[
\mathcal{M}(B,2t)\leq \mathcal{N}(\bar{B},t)= \mathcal{N}(B,t) \leq \mathcal{M}(B,t)=\mathcal{M}(\bar{B},t).
\]
\end{prop}

%

The idea to make the notion of size depends on the packing of the set is not random and comes from the fact that packing appears in many notions of size. 
For instance, the packing measure defined in \cite{saint1988} is another definition of the Lebesgue measure (up to some constant factor) and is in the more general case of a metric space an isometric measure that coincide with Hausdorff measure on spaces with non fractional dimension as shown in \cite{saint1988}.

We will then study size function $\tau$ of the form
\[
\tau(B)=\Phi(\mathcal{M}(B,.)),
\]
for all $B\in \mathcal{F}$ for a function $\Phi$ on the set of packing functions.

In order to ensure that the condition (H\ref{d-cont}) is fulfilled, we define the following hypotheses.\\[1ex]

\noindent (H'\ref{increasing}) $\tau$ is strictly increasing (i.e. $A\subset B, A\neq B \implies \tau(A)<\tau(B)$),\\[2ex]
(H'\ref{lsc}) for any sequence $(B_n)_{n\geq 1} \subset \mathcal{F}$, $\Phi(\liminf \mathcal{M}(B_n,.)) \leq \liminf \Phi(\mathcal{M}(\lim B_n,.))$,

\begin{enumerate}
\setcounter{enumi}{\value{compteur}}
\item $\tau(A)< \infty \implies A \text{ totally bounded}.$ \label{taucompact}
\setcounter{compteur}{\value{enumi}}
\end{enumerate}

Hypotheses (H'\ref{increasing}) and (H'\ref{lsc}) imply respectively (H\ref{increasing}) and (H\ref{lsc}) and we can show that, together with (H\ref{taucompact}), they imply (H\ref{d-cont}), so that we have the following theorem.

\begin{thm}[Packing size function]\label{thm:packing}
Suppose that $\tau$ is a size function of the form $\tau(B)=\Phi(\mathcal{M}(B,.))$ on a stable class $\mathcal{F}$. Suppose (H'\ref{increasing}), (H\ref{monotony}), (H'\ref{lsc}), (H\ref{taucont}) and (H\ref{taucompact}).

Then, almost surely,
\[
\Haus(\mathcal{S}^\alpha_n|\mathcal{S}^\alpha)\rightarrow 0.
\]
\end{thm}
%
%

\begin{rem}[On (H\ref{taucompact})]
Since a non compact space can not be arbitrarily close to a finite set in Hausdorff distance, then for $B_n^\alpha$ to converge when $(\mu_n)_{n \geq 1}$ is sequence of finitely supported measures, minimizers relative to $\mu$ must be compact if $\mathcal{F}$ is rich enough to make minimizers finite sets. Minimizers relative to $\mu$ are compact when $\tau(B) < \infty$ implies that $B$ is totally bounded.
\end{rem}

\subsection{Examples of sequence of measures}

Proposition \ref{tauemp} and theorem \ref{Hausdorffconsistency} can be applied to the empirical measure and then lead to the following.

\begin{thm}[Consistency for empirical measure: i.i.d. case]\label{thm:cas_iid}
Let $(E,d)$ be a Polish space, $\mu$ a probability measure on $(E,\mathbb{B}(E))$ and $\mathcal{F}$ a stable set such that $\tau^\alpha<\infty$, for any $0 < \alpha < 1$.
Given a sample $(X_n)_{n\geq 1}$ of independent random variables with same law $\mu$, set
\[
\mu_n=\frac{1}{n}\sum_{1\leq i \leq n} \delta_{X_i},
\]
the empirical measure. Suppose (H\ref{monotony}), (H\ref{lsc}), (H\ref{taucont}) and (H\ref{d-cont}).
Then, almost surely,
\[
\Haus(\mathcal{S}^\alpha_n|\mathcal{S}^\alpha)\rightarrow 0.
\]
\end{thm}

Likewise, corollary \ref{cor:tau_tension} leads to convergence of empirical measure for some dependent cases.

\begin{thm}[Consistency for empirical measure: dependent case]\label{thm:cas_dependant}
Let $(E,d)$ be a Polish space and $\mu$ be a probability measure on $(E,\mathbb{B}(E))$ such that $\tau^\alpha<\infty$, for any $0 < \alpha < 1$. Given random variables $(X_n)_{n\geq 1}$ from an ergodic Markov chain with invariant measure $\mu$, denote
\[
\mu_n=\frac{1}{n}\sum_{1\leq i \leq n} \delta_{X_i},
\]
the empirical measure. Suppose (H\ref{monotony}), (H\ref{lsc}), (H\ref{taucont}), (H\ref{d-cont}) and (H\ref{taucompact}). Then, almost surely,
\[
\Haus(\mathcal{S}^\alpha_n|\mathcal{S}^\alpha)\rightarrow 0.
\]
\end{thm}

\section{Proofs}

The following lemma is the starting point of the existence of minimizer theorem.
It is also a key lemma to most of the other results.

\begin{lem}\label{limcomp}
Let $(E,d)$ be a Polish space. Let $(\mu_n)_{n \geq 1}$ be a sequence of probability measures weakly converging to $\mu_\infty$. Then,
\[
\supp \mu_\infty \subset \lim_n \supp \mu_n.
\]
Moreover, for any $t>0$
\[
\mathcal{M}(\supp \mu_\infty,t) \leq \mathcal{M}(\lim_n \supp \mu_n,t) \leq \liminf_n \mathcal{M}(\supp\mu_n,t).
\] 
\end{lem}
\begin{proof}[Proof of lemma \ref{limcomp}]
Set $t>0$ and $x \in \supp \mu_\infty$. Then, for all $\eps>0$, there exists $\eta>0$ such that $\mu_\infty(B(x,\eps/2))>\eta$. Then, using the portmanteau theorem, for $n$ large enough, $\mu_n(B(x,\eps/2))>\eta$ and thus, $x \in (\supp \mu_n)^\eps$ for any $n$ large enough, which proves the first point and first inequality.

For the second inequality, choose $m \leq \mathcal{M}(\lim_n \supp \mu_n,t)$.
Let $\{x_i\}_{1 \leq i\leq m}$ be a $t$-separated subset of $\lim_n \supp \mu_n$. Then, there exists $\delta>0$ such that 
\[
\inf_{i \neq j} d(x_i,x_j) > t+\delta.
\]
From definition of $\lim_n \supp \mu_n$ and since $m$ is finite, there exists $k \geq 1$ such that $(x_i)_{1 \geq i \geq m} \subset \cap_{n\geq k} (\supp \mu_n )^{\delta/2}$

Choose now $x_i^n \in B(x_i,\delta/2)\cap \supp \mu_n$ for $1 \leq i \leq m$.
$\{x_i^n\}_{1 \leq i\leq m}$ forms a $t$-separated set included in $\supp \mu_n$ since for $i \neq j$
\begin{align*}
d(x_i^n,x_j^n) &\geq d(x_i,x_j)-d(x_i,x_i^n)-d(x_j,x_j^n)\\
& > t +\delta -\delta/2 - \delta/2 =t.
\end{align*}
Thus, $m \leq \liminf_n \mathcal{M}( \supp\mu_n,t)$, for any $m \leq \mathcal{M}( \lim_n \supp \mu_n,t)$, 
which ends the proof.
\end{proof}

\begin{proof}[Proof of theorem \ref{existence}]
Let $(B_n)_{n \geq 1}\subset \mathcal{F}$ be a minimizing sequence such that $\mu(B_n)\geq 1-\alpha$ for all $n\geq 1$. Set $\mu_n=\mu(.|B_n)$. Let first show that $(\mu_n)_{n\geq 1}$ is tight. Let $K$ be a compact set such that $\mu(K)\geq 1-\eps(1-\alpha)$, then
\begin{align*}
\mu_n(K)&=\mu(K\cap B_n)/\mu(B_n)\\
&= (1-\mu(K^c\cup B_n^c))/\mu(B_n)\\
&\geq 1-\mu(K^c)/\mu(B_n)\\
&\geq 1-\eps.
\end{align*}
Thus, up to a subsequence, $(\mu_n)_{n\geq 1}$ weakly converges to $\mu_\infty$. Set $B=\lim_n \supp \mu_n$. $B^\eps$ is an open set and contains $\supp \mu_\infty$ from lemma \ref{limcomp}. Then, portmanteau theorem gives
\begin{align*}
1 = \mu_\infty(B^\eps) \longleftarrow \mu_n(B^\eps) & = \mu(B^\eps \cap B_n)/\mu(B_n)\\
&\leq \mu(B^\eps)/(1-\alpha).
\end{align*}
Letting $\eps$ tends to zero and using $\bigcap_{\eps>0} B^\eps=B$, it comes $\mu(B)\geq 1-\alpha$.

Proof ends using (H\ref{lsc}) and stability of $\mathcal{F}$.
\end{proof}

\begin{proof}[Proof of proposition \ref{rightcont}]
Clearly, $\alpha \mapsto \tau^\alpha$ decreasing, and we thus only need to show $\tau^{\alpha}\geq \lim_{\eps \rightarrow 0^+} \tau^{\alpha+\eps}$.
Set $B_\alpha \in \argmin \{\tau(A); A \in \mathcal{F}, \mu(A)\geq 1-\alpha\}$ and $\mu^\alpha=\mu(.|B_\alpha)$. For any $(\alpha_n)_{n \geq 1} \subset (0;1)$ decreasing and converging to $\alpha$, $(\mu^{\alpha_n})_{n\geq 1}$ is tight. Indeed, for a compact set $K$ such that $\mu(K) \geq 1-\eps(1-\alpha_1)$,
\begin{align*}
\mu^{\alpha_n}(K) &= \mu(K\cup B_{\alpha_n})/\mu(B_{\alpha_n})\\
&\geq (1-\mu(K^c)-\mu(B_{\alpha_n}^c)/\mu(B_{\alpha_n})\\
&=1-\mu(K^c)/\mu(B_{\alpha_n})\\
&\geq 1-\eps.
\end{align*}
Thus, up to a subsequence, $\mu^{\alpha_n}$ converges to some probability measure $\mu_\infty$. Set $B_\infty= \lim_n\supp \mu_n$. $B_\infty^\eps$ is an open set that contains $\supp \mu_\infty$, and thus portmanteau theorem yields, 
\begin{align*}
1 = \mu_\infty(B_\infty^\eps) \longleftarrow& \mu^{\alpha_n}(B_\infty^\eps) \\
= &\mu(B_\infty^\eps \cap B_{\alpha_n})/\mu(B_{\alpha_n})\\
\leq &\mu(B_\infty^\eps)/(1-\alpha_n)\longrightarrow  \mu(B_\infty^\eps)/(1-\alpha).
\end{align*}
Letting $\eps$ tends to zero shows $\mu(B_\infty) \geq 1-\alpha$.
Hypothesis (H\ref{lsc}) and stability of $\mathcal{F}$ let us conclude
\[
\tau^\alpha \leq \tau(B_\infty) \leq \liminf_n \tau(B_{\alpha_n})=\lim_n \tau(B_{\alpha_n}).
\]
\end{proof}

In order to show continuity of the size function (proposition \ref{continuity}), we establish a lemma that states that for a connected measure, any $\eps$-neighborhood of Borel set has strictly more mass than the original Borel set.

\begin{lem}\label{connected}
Let $\mu$ be a probability measure with a connected support, then for any Borel set $A$ such that $\mu(A)<1$ and any $\eps >0$, the $\eps$-neighborhood of $A$ satisfies
\[
\mu(A^\eps) > \mu(A).
\]
\end{lem}
\begin{proof}[Proof of lemma \ref{connected}]
In order to show a contradiction, suppose $\mu(A^\eps)=\mu(A)$ for some $\eps >0$. Then, $\mu(A^{\eps/3})=\mu(A)$ and $\mu(((A^\eps)^c)^{\eps/3}) \geq \mu((A^\eps)^c)=1-\mu(A)$ so that $\mu(A^{\eps/3} \cup ((A^\eps)^c)^{\eps/3})=1$. Since $\overline{A^{\eps/3}} \cap \overline{((A^\eps)^c)^{\eps/3}}= \emptyset$ and $\overline{A^{\eps/3}}$ and $\overline{((A^\eps)^c)^{\eps/3}}$ are closed, the support of $\mu$ not connected, which contradicts the hypothesis.
\end{proof}

We can now prove continuity of the smallest size function.

\begin{proof}[Proof of proposition \ref{continuity}]
By proposition \ref{rightcont} it is enough to show that $\alpha \mapsto \tau^\alpha$ is left continuous. Moreover, since $\alpha \mapsto \tau^\alpha$ is decreasing, it is enough to prove $\lim_{\eps\rightarrow 0^+} \tau ^{\alpha-\eps} \leq \tau^\alpha$. 

Let $\alpha>0$ and $B\in \mathcal{F}$ such that $\tau^\alpha=\tau(B)$. If $\mu(B)>1-\alpha$ then the result is obvious.
Suppose then $\mu(B)=1-\alpha$. By lemma \ref{connected}, for any $\delta>0$ there exists $\eps>0$ such that $\mu(B^\delta)\geq 1-\alpha+\eps$.
Define a decreasing sequence $(\delta_n)_{n \geq 1}$ converging to zero, such that $\mu(B^{\delta_n}) > 1-\alpha$ for any $n \geq 1$.
Since $\mu$ is $\mathcal{F}$-regular, there exists $(K_n)_{n\geq 1} \subset \mathcal{F}$ such that for any $n\geq 1$
\begin{align*}
&\tau(K_1)<\infty, \text{ since } \tau^{\alpha-\eps}<\infty, \\
&K_n \subset B^{\delta_n} \cap K_{n-1} ,\\
&1-\alpha<\mu(K_n) =: 1-\alpha_n.
\end{align*}
Clearly, $\alpha_n \rightarrow \alpha$.
Then, 
using (H\ref{increasing}) and (H\ref{monotony}), it yields
\[
\lim_{\eps \rightarrow 0^+} \tau^{\alpha-\eps} \longleftarrow \tau^{\alpha_n} \leq \tau(K_n) \longrightarrow \tau(\cap_n K_n)= \tau^\alpha.
\]
\end{proof}

\begin{proof}[Proof of proposition \ref{tauemp}]
Since $\tau^\alpha$ is finite, for any $0<\alpha<1$, there exists $B \in \mathcal{F}$ such that $\tau(B) < \infty$ and $\mu(B)\geq 1-\alpha$.
For such $B$, the law of large number states that almost surely $\lim_n\mu_n(B)=\mu(B)$.
Set $\delta>0$. Then for any $C\in \mathcal{F}$ such that $B\subset C$ and $ \mu(B)\leq\liminf_n \mu_n(C)$, for $A:=B \in \mathcal{F}$ the three following conditions are clearly true
\begin{align*}
\mu(B) \leq \liminf_n \mu_n(A), \\
B \subset A \subset B^\delta \cap C,\\
\tau(A)<\infty.
\end{align*}
\end{proof}

\begin{proof}[Proof of theorem \ref{consistency}]
Let $K$ be a compact set such that $\mu_n(K)\geq 1-\eps(1-\alpha)$ for any $n\geq 1$ and $\mu(K)\geq 1-\eps(1-\alpha)$. Then, 
\begin{align*}
\mu_n^\alpha(K) &= \mu_n(K\cap B_n^\alpha)/\mu_n(B_n^\alpha)\\
&\geq 1-\eps,
\end{align*}
showing tightness of $(\mu_n^\alpha)_{n\geq 1}$.
Denote $\mu_\infty^\alpha$ a limit measure and $B_\infty^\alpha = \lim_n \supp \mu_n^\alpha$.
Set $\eps=\inf\{\eta >0 ; \mu_n(B) \leq \mu(B^\eta) + \eta, \forall B \in \mathbb{B}(E)\}$ the Prokhorov distance between $\mu$ and $\mu_n$.
Then using together the portmanteau theorem and the fact that the topology induced by the Prokhorov metric coincides with the weak convergence topology, for any $\eta>0$ there exists $n_\eta$ such that for all $n>n_\eta$ - since $\supp \mu_\infty^\alpha \subset B_\infty^\alpha$ from lemma \ref{limcomp},
\begin{align*}
1 &= \mu_\infty^\alpha (B_\infty^\alpha)\\
&\leq \mu_n^\alpha ((B_\infty^\alpha)^\eps|B_n^\alpha) + \eta\\
&=\mu_n((B_\infty^\alpha)^\eps\cap B_n^\alpha)/\mu_n(B_n^\alpha) + \eta\\
&\leq \mu_n((B_\infty^\alpha)^\eps)/(1-\alpha)+ \eta\\
&\leq (\mu((B_\infty^\alpha)^{2\eps})+\eps)/(1-\alpha) + \eta \longrightarrow \mu(B_\infty^\alpha)/(1-\alpha).
\end{align*}
Thus, $\mu(B_\infty^\alpha) \geq 1-\alpha$ and thus by definition of $\tau^\alpha$,
\[
\tau^\alpha \leq \tau(B_\infty^\alpha).
\] 

The second inequality is a direct application of (H\ref{lsc}).

Let us relabel a subsequence of $(\mu_n)_{n \geq 1}$ so that $\limsup_{n \rightarrow \infty} \tau(B_n^\alpha) = \lim_{n \rightarrow \infty} \tau(B_n^\alpha)$.
Choose $\eps>0$, and $B_\eps \in \argmin\{\tau(A); A \in \mathcal{F}, \mu(A)\geq 1-\alpha +\eps\}$. Without loss of generality, we can suppose that $\tau(B_\eps)<\infty$. Since the sequence $(\mu_n)_{n\geq1}$ converges to $\mu$, the portmanteau theorem states that for any $\delta >0$, 
\[
\liminf_n\mu_n(B_\eps^\delta)\geq \mu(B_\eps).
\]
Set $\delta_1>0$. Then, since $E\in \mathcal{F}$, $\tau$-tightness implies that there exists $A_1 \in \mathcal{F}$ such that
\begin{align*}
\tau(A_1)< \infty,\\
B_\eps \subset A_1 \subset B_\eps^{\delta_1},\\
1-\alpha +\eps \leq \mu(B_\eps) \leq \liminf_n\mu_n(A_1), \text{ a.s.}.
\end{align*}

Let $(\delta_k)_{k\geq 1}$ be a decreasing sequence converging to zero.
Using recursively the $\tau$-tightness property of the sequence $(\mu_n)_{n\geq 1}$, with sets $A_{k-1}$, define a decreasing sequence $(A_k)_{k\geq 1}$ of elements of $\mathcal{F}$, such that $A_k \subset B_\eps^{\delta_k}$ for any $k\geq 1$, and
\begin{align*}
B_\eps \subset A_k \subset B_\eps^{\delta_k},\\
1-\alpha +\eps \leq \mu(B_\eps) \leq \liminf_n\mu_n(A_k) \text{ a.s.}.
\end{align*}
Then, for any $k\geq 1$, there exists $n_k \geq 1$ such that $\mu_{n}(A_k) \geq 1-\alpha$ for any $n>n_k$. Then, since $A_k \in \mathcal{F}$, $\tau(B_n^\alpha) \leq \tau(A_k)$ for all $n\geq n_k$.
By construction and since $B_\eps$ is closed, $\cap_k A_k = B_\eps$.
Using (H\ref{increasing}) and (H\ref{monotony}) yields
\[
\tau(B^\alpha_n) \leq \tau(A_k) \longrightarrow_k \tau(\cap_k A_k) = \tau(B_\eps),
\]
which shows that for any $\eps >0$,
\[
\limsup_n \tau(B_n^\alpha) \leq \tau^{\alpha-\eps}.
\]

The inequality of the five terms is a direct application of proposition \ref{continuity}.
\end{proof}

\begin{proof}[Proof of theorem \ref{Hausdorffconsistency}]
Choose $B_n^\alpha \in \mathcal{S}_n^\alpha$. Set $\mu_n^\alpha = \mu_{n|B_n^\alpha}$. Theorem \ref{consistency}, lemma \ref{limcomp} and continuity condition (H\ref{taucont}) of $\cdot \mapsto \tau^\cdot$ shows that $(\mu_n^\alpha)_{n \geq 1}$ is totally bounded for the weak convergence topology and that $B_\infty^\alpha:=\lim_{k} B_{n_k} \in \mathcal{S}^\alpha$ where $(n_k)_{k \geq 1}$ is a subsequence along which $\mu_n^\alpha$ converges.
%
Theorem \ref{consistency} and (H\ref{d-cont}) then yield
\begin{equation}\label{firstfirstdH}
d_H(B^\alpha_\infty , B_n^\alpha) \rightarrow 0,
\end{equation}
for the subsequence of $(B_n)_{n \geq 1}$.
Since (\ref{firstfirstdH}) holds for a subsequence of any subsequence of the original sequence $(B_n^\alpha)_{n \geq 1}$, taking $B_n^\alpha \in \mathcal{S}^\alpha_n$
so that $1/n+d_H(B_n^\alpha,B_\infty^\alpha) \geq \Haus(\mathcal{S}^\alpha_n|\mathcal{S}^\alpha)$ for any $B_\infty^\alpha \in \mathcal{S}^\alpha$, it follows
\[
\Haus(\mathcal{S}^\alpha_n|\mathcal{S}^\alpha)\rightarrow 0.
\]
\end{proof}

\begin{proof}[Proof of proposition \ref{contHaus}]
Choose a sequence $(\alpha_n)_{n \geq 1}$ converging to $\alpha$.
From continuity of the size function, for any $n$ large enough, $\tau^{\alpha_n} < \infty$. 
Define $\mu_n=\mu(.|B_n)$ for any $B_n \in \mathcal{S}^{\alpha_n}$. Sequence $(\mu_n)_{n\geq 1}$ is tight.
Indeed, let $K$ be a compact set such that $\mu(K)\geq 1-\eps(1-\max_{n \geq 1} \alpha_n)$, then
\begin{align*}
\mu_n(K)&=\mu(K\cap B_n)/\mu(B_n)\\
&= (1-\mu(K^c\cup B_n^c))/\mu(B_n)\\
&\geq 1-\mu(K^c)/\mu(B_n)\\
&\geq 1-\eps.
\end{align*}
Set $B_\infty=\lim_k B_{n_k}$  with $(n_k)_{k \geq 1}$ a subsequence along which $(\mu_n)_{n \geq 1}$ converges (such subsequence exists since $(\mu_n)_{n\geq 1}$ is tight). Set $\mu_\infty = \lim_k \mu_{n_k}$. Using similar arguments to those in theorem \ref{existence}, we can show that $\mu(B_\infty) \geq 1 - \alpha$. Thus hypothesis (H\ref{lsc}) yields
\begin{equation}\label{contcons}
\tau^\alpha \leq \tau(B_\infty) \leq \liminf \tau(B_n) = \lim \tau^{\alpha_n} = \tau^\alpha,
\end{equation}
where the last equality is due to the continuity condition (H\ref{taucont}).
(H\ref{d-cont}) implies
\begin{equation}\label{firstdH}
d_\mathcal{F}(B_\infty,B_n) \rightarrow 0.
\end{equation}
Remark then that corollary \ref{corHausdorffconsistency} establishes that $\mathcal{S}^{\alpha_n}$ is compact for the Hausdorff metric topology, with underlying metric $d_\mathcal{F}$. Since (\ref{firstdH}) holds for a subsequence of any subsequence of the original sequence $(B_n)_{n \geq 1}$, taking
\[
B_n \in \argmax\{\min_{B \in \mathcal{S}^\alpha}d_H(B,A)|A\in \mathcal{S}^{\alpha_n}\},
\]
so that $d_H(B_n,B_\infty) \geq \Haus(\mathcal{S}^{\alpha_n}|\mathcal{S}^\alpha)$ for any $B_\infty \in \mathcal{S}^\alpha$ yields
\[
\Haus(\mathcal{S}^{\alpha_n} | \mathcal{S}^\alpha)\rightarrow 0.
\]
\end{proof}

\begin{proof}[Proof of corollary \ref{corHausdorffconsistency}]
Set $(\mu_n)_{n\geq 1}=(\mu)_{n\geq1}$ in the theorem $\ref{Hausdorffconsistency}$ and the result is then a direct application of (\ref{firstfirstdH}) from the proof of the proposition.
\end{proof}

\begin{proof}[Proof of corollary \ref{thm:packing}]
In order to prove this result, we show that (H'\ref{increasing}), (H'\ref{lsc}) and (H\ref{taucompact}) all together imply (H\ref{increasing}), (H\ref{lsc}) and (H\ref{d-cont}).
Thus, suppose (H'\ref{increasing}), (H'\ref{lsc}) and (H\ref{taucompact}).
Then, (H\ref{increasing}) and (H\ref{lsc}) are obviously true.

Now, in order to show  (H\ref{d-cont}), choose  a sequence $(A_n)_{n\geq 1} \subset \mathcal{F}$ such that $\tau(\lim_n A_n)<\infty$, and $\lim\tau(A_n) = \tau(\lim A_n)$. Then, lemma \ref{lemHausdorff} implies that 
\[
H(\lim_kA_k|A_n) \rightarrow 0.
\]
Moreover,
\begin{align*}
\lim_n \tau(A_n) & = \lim_n \Phi(\mathcal{M}(A_n,.))\\
& \geq \Phi(\liminf \mathcal{M}(A_n,.) \text{ using  (H'\ref{lsc})}\\
& \geq \Phi(\mathcal{M}(\lim_nA_n,.)) \text{ using lemma \ref{limcomp}}\\
& = \tau(\lim A_n).
\end{align*}
Thus, since we supposed $\lim\tau(A_n) = \tau(\lim A_n)$, and because $\Phi$ is supposed to be strictely increasing by (H'\ref{increasing}), lemma \ref{limcomp} shows
\[
\liminf\mathcal{M}(A_n,t)=\mathcal{M}(\lim A_n,t), \forall t>0.
\]
Lemma \ref{contrasteHaus} can thus be applied to show that, along a subsequence,
\[
d_H(A_n,\lim_k A_k)\rightarrow 0.
\]
Since this is true for any subsequence, it holds for the sequence itself, i.e. (H\ref{d-cont}) holds.

The result is then an application of theorem \ref{Hausdorffconsistency}.
\end{proof}


\begin{lem}\label{lemHausdorff}
Let $(E,d)$ be a Polish space and $(B_n)_{n \geq 1}$ be a sequence of set of $E$. Then, for any $\eps>0$ such that $\mathcal{M}(\lim_kB_k,\eps) < \infty$, there exists $n_0 \in \mathbb{N}$ such that for all $n>n_0$,
\[
\lim_k B_k \subset B_n^{2\eps}.
\]
\end{lem}
\begin{proof}[Proof of lemma \ref{lemHausdorff}]
Set $\eps>0$ and $m= \mathcal{M}(\lim_k B_k,\eps)$.
Let $\{x_i\}_{1 \leq i\leq m}$ be a maximal $\eps$-separated subset of $\lim_n B_n$. Then, by definition of $\lim_k B_k$ and since $m$ is finite, for any $\delta>0$, there exists $k_\delta$ such that for all $n\geq k_\delta$,
\[
\{x_i\}_{1 \leq i\leq m} \subset B_n^\delta.
\]
Maximality of $\{x_i\}_{1 \leq i\leq m}$ ensures
\[
\lim_k B_k \subset \bigcup_{1 \leq i \leq m} \overline{B(x_i,\eps)} \subset B_n^{\eps+\delta}.
\]
Proofs ends when choosing $\delta \leq \eps$.
\end{proof}

\begin{lem}\label{contrasteHaus}
Let $(E,d)$ be a metric space. Let $(B_n)_{n\geq 1}$ be a sequence of closed sets such that $\Haus(B|B_n) \rightarrow 0,$ for some closed set $B$.
Suppose that for any $s >0$, there exists $0<t< s$ such that
\begin{equation}\label{Mhypothesis}
\liminf\mathcal{M}(B_n,t) \leq \mathcal{M}(B,t) < \infty.
\end{equation}
Then along a (relabeled) subsequence $(B_n)_{n \geq 1}$,
\[
d_H(B_n,B) \rightarrow 0.
\]
\end{lem}
\begin{proof}[Proof of lemma \ref{contrasteHaus}]
It suffices to show
\[
\forall \eps > 0, \forall n_0 \in \mathbb{N}, \exists n\geq n_0, B_n \subset (B)^\eps.
\]
In order to prove a contradiction, suppose
\begin{equation}\label{falsestatement}
\exists \eps > 0, \exists n_0 \in \mathbb{N}, \forall n\geq n_0, B_n \not\subset (B)^\eps.
\end{equation}
Then,
\[
\exists \eps>0, \exists n_0 \in \mathbb{N}, \forall n>n_0, \exists x_n \in B_n, \forall y \in B, d(x_n,y)>\eps.
\]
Set $t<\eps$ and $m=\mathcal{M}(B,t)$. Since $\mathcal{M}(B,.)$ is right continuous, there exists $\delta > 0$ such that 
\[
m=\mathcal{M}(B,t) = \mathcal{M}(B,t+\delta).
\]
Let $\{z_i\}_{i \leq m}$ be a maximal $(t+\delta)$-separated subset of $B$.
Then, since for any $\eta>0$ and $n$ large enough,
\[
B \subset (B_n)^{\eta},
\]
we can construct $\{y_i\}_{i \leq m} \subset B_n$ such that for all $1 \leq i \leq m$, $d(y_i,z_i)\leq \eta$.
Thus, for any $i \neq j$, 
\[
d(y_i,y_j) \geq d(z_i,z_j)-d(z_i,y_i)-d(z_j,y_j) > t+\delta - 2 \eta.
\]
Choose then $\eta < \delta/2 \wedge (\eps-t)$ such that $\{y_i\}_{i \leq m}$ is a $t$-separated subset of $B_n$.
Then, $\{x_n\}\cup \{y_i\}_{i \leq m}$ is also a $t$-separated subset of $B_n$ since $x_n \in B_n$ and
\[
d(x_n,y_i) \geq d(x_n,z_i) -d(z_i,y_i)>\eps - \eta > t,
\]
showing thus $\mathcal{M}(B,t) +1 \leq \mathcal{M}(B_n,t)$, for any $n$ large enough. This contradicts (\ref{Mhypothesis}) and thus shows that (\ref{falsestatement}) is false.
\end{proof}


\begin{lem}\label{lem:stable}
Let $\mathcal{F}$ be a stable class of connected sets of a metric space $(E,d)$ such that bounded sets are compact. Suppose that there exists an increasing union $(K_k)_{k \geq 1}$ of balls verifying $\cup_{k \geq 1} K_k=E$ such that for all $k\geq 1$ and any $F \in \mathcal{F}$, $K_k\cap F$ is still connected then
\[
\mathcal{F}_\eps := \{ \cup_{F \in \mathcal{G}} F ; \mathcal{G} \subset \mathcal{F}, \forall F,G \in \mathcal{G} \inf_{x \in F, y \in G} d(x,y) \geq \eps \},
\]
is also a stable class.
\end{lem}
\begin{proof}
Set $(B_n)_{n \geq 1} \subset \mathcal{F}_\eps$. Remark \ref{rem:totalborne} states that there exists $B_\infty$ such that a subsequence of $(B_n)_{n\geq 1}$ converges to $B_\infty$ for the generalized Hausdorff convergence. Suppose that there exist two connected components $F_1, F_2 \subset B_\infty$ such that
\[
\underline{d}(F_1,F_2):=\inf \{d(x,y); x \in F_1, y \in F_2 \}<\eps.
\]
Then, the hypothesis on $E$ implies that for any ball $K$ for which $K\cap F_1 \neq \emptyset$ and $K\cap F_2 \neq \emptyset$, there exists $x_1 \in F_1\cap K$ and $x_2 \in F_2 \cap K$ such that $\underline{d}(F_1\cap K,F_2\cap K) = d(x,y)$.
One can check that if $x$ and $y$ belong to the interior of $K$ (which we can assume without loss of generality), it holds
\[
d_H(B_n\cap K, B_\infty \cap K) \rightarrow 0.
\]
Taking then $K\in \{K_k\}_{k\geq 1}$ large enough so that
\[
\underline{d}(F_1\cap K,F_2 \cap K)=\eta<\eps.
\]
and setting
\[
\delta=\underline{d}(K\cap F_1,K\cap B_\infty \setminus F_1)\wedge\underline{d}(K\cap F_1,K\cap B_\infty \setminus F_1)>0,
\]
yields for all $n$ large enough
\begin{align}
F_1\cap K \subset B_n^{((\eps-\eta)\wedge\delta)/3} \cap K &\text{ and } F_2 \cap K \subset B_n^{((\eps-\eta)\wedge \delta)/3},\label{eq:contr}\\
B_n \cap K \subset & B_\infty^{\delta/2} \cap K.\label{eq:connected}
\end{align}
(\ref{eq:contr}) shows then that there exist $y_1,y_2 \in B_n\cap K$ such that $d(x_1,y_1)\leq ((\eps-\eta)\wedge \delta)/3$ and $d(x_2,y_2)\leq ((\eps-\eta)\wedge\delta)/3$ and so
\begin{align}
d(y_1,y_2) & \leq d(y_1,x_1) + d(x_1,x_2) + d(x_2,y_2)\nonumber\\
&\leq 2(\eps-\eta)/3 + \eta\nonumber\\
&\leq 2\eps/3 + \eta /3 < \eps.\label{eq:concl}
\end{align}
And (\ref{eq:connected}) shows that $B_n \cap K \cap F_1^{\delta/2} \subset F_1^{\delta/4} \cap K$ and $B_n \cap K \cap F_2^{\delta/2} \subset F_2^{\delta/4} \cap K$, since $K\cap B_\infty^{\delta/4} \cap F_i^{\delta/2}\subset K \cap F_i^{\delta/4}$ for $i=1,2$ by definition of $\delta$.
It follows that $y_1 \in B_n \cap K \cap F_1^{\delta/2}$ and $y_2 \in B_n \cap K \cap F_2^{\delta/2}$ are included in two distinct connected components $B_n\cap K$, which contradicts (\ref{eq:concl}) and hypothesis that $B_n\in \mathcal{F}_\eps$.

It is then easily checked that $F_1,F_2 \in \mathcal{F}$.
\end{proof}

\section{Conclusion}

We have thus defined a new way to localize mass of measure that is consistent for empirical measures, in Hausdorff metric. It is rather intuitive and applies to any Polish space, including infinite dimensional spaces. It thus provides an analogy to level sets in theses spaces.
The major drawbacks of our methods lie in the computability of the size function for rich classes $\mathcal{F}$ and the lack of rate of convergence for now.

\bibliography{mass_loc}

\end{document}